\documentclass[12pt,a4paper,psamsfonts]{amsart}
\usepackage[english]{babel}
\usepackage[T1]{fontenc}
\usepackage{mathtools}
\usepackage{fancyhdr}
\usepackage{appendix}
\usepackage{amssymb,amscd,amsxtra,calc}
\usepackage{mathrsfs}
\usepackage{cmmib57}
\usepackage{multirow}
\usepackage[all]{xy}
\usepackage{longtable}
\usepackage[colorlinks=true,linkcolor=blue,anchorcolor=blue,citecolor=blue,pagebackref,linktocpage]{hyperref}
\usepackage{cleveref}
\usepackage{tikz}
\usepackage{cite}
\theoremstyle{plain}
    \newtheorem{thm}{Theorem}[section]

    \newtheorem{lemma}[thm]{Lemma}
    \newtheorem{proposition}[thm]{Proposition}
    \newtheorem{question}[thm]{Question}
    \newtheorem{theorem}[thm]{Theorem}
    \newtheorem{problem}[thm]{Problem}

\theoremstyle{definition}
    \newtheorem{definition}[thm]{Definition}
     
    \newtheorem{assumption}[thm]{Assumption}
    
    \newtheorem*{notation*}{Notation and Terminology}
      
    \newtheorem{remark}[thm]{Remark}
    
\theoremstyle{remark}

\newcommand{\bP}{\mathbb{P}}
\newcommand{\bQ}{\mathbb{Q}}

\newcommand{\bC}{\mathbb{C}}

\newcommand{\bZ}{\mathbb{Z}}

\newcommand{\mstriangle}[1]{
\begin{tikzpicture}[x=0.3cm,y=0.3cm]
\draw (-0.4,-0.433) -- (1.4,-0.433);
\draw (-0.2,-0.7794) -- (0.7,0.7794);
\draw (1.2,-0.7794) -- (0.3,0.7794);
\end{tikzpicture}
}
\newcommand{\mssharp}[1]{
\begin{tikzpicture}[x=0.3cm,y=0.3cm]
\draw (-0.8,-0.5) -- (0.8,-0.5);
\draw (-0.8,0.5) -- (0.8,0.5);
\draw (-0.5,-0.8) -- (-0.5,0.8);
\draw (0.5,-0.8) -- (0.5,0.8);
\end{tikzpicture}
}

\makeatletter

\newcommand{\Rmnum}[1]{\expandafter\@slowromancap\romannumeral #1@}
\makeatother

\title[$(-1)$-curves with higher genus]{Examples of infinitely many $(-1)$-curves with higher genus on elliptic surfaces}

\author{Sichen Li}
\address{
School of Mathematics, East China University of Science and Technology, Shanghai 200237, P. R. China}
\email{\href{mailto:sichenli@ecust.edu.cn}{sichenli@ecust.edu.cn}}
\author{Jihao Liu}
\address{
Department of Mathematics, Peking University, No. 5 Yiheyuan Road, Haidian District, Beijing 100871, China
\endgraf 
Beijing International Center for Mathematical Research, Peking University, No. 5 Yiheyuan Road, Haidian District, Beijing 100871, China}
\email{\href{mailto:liujihao@math.pku.edu.cn}{liujihao@math.pku.edu.cn}}

\begin{document}

\begin{abstract}
In this article, for every integer $g>1$, we construct the first example of an elliptic surface with infinitely many smooth \((-1)\)-curves of genus \(g\), settling an open question of Bauer et al. [Duke Math. J. \textbf{162} (10) (2013), 1877-1894].
\end{abstract}

\maketitle
\section{Introduction}
\subsection{Negative curves with higher genus}
It is classical that blowing up $\mathbb{P}^2$ at nine very general points yields a rational surface containing infinitely many genus-zero $(-1)$-curves. A parallel construction produces surfaces admitting infinitely many genus-one $(-m)$-curves for every integer $m \geq -1$: one blows up torsion points along a constant-torsion family of elliptic curves on the abelian surface $A = E \times E$, where $E$ is an elliptic curve without complex multiplication (see \cite[Theorem 4.1]{Bauer et al. 2013}). These examples motivate the following highly intriguing open question:

{\bf Question.}
For each integer $m \geq 1$ and each $g \geq 0$, does there exist a smooth projective surface carrying infinitely many $(-m)$-curves of genus $g$?

Bauer et al.\ \cite[Theorem 4.3]{Bauer et al. 2013} give an affirmative answer to the above question for all $m > 1$ and arbitrary $g \geq 0$. Their reasoning proceeds as follows: starting from an elliptic rational surface with infinite Mordell--Weil group, one carries out a degree-$m$ base change with a  curve of genus $g$ to obtain a surface with $\chi(\mathcal O_Y) = m$. The Mordell--Weil group remains infinite, which implies that for every $m \geq 1$ and $g \geq 0$, there exists an elliptic surface with infinitely many  $(-m)$-curves of genus $g$.

The case $m=1$ and $g>1$ remains open and unresolved for more than a decade. Precisely, Bauer et al.\ formulate this specialized open question:
\begin{question}
\label{MainQue}
\cite[Question 4.5]{Bauer et al. 2013}
For every integer $g > 1$, does there exist a smooth projective surface admitting infinitely many $(-1)$-curves of genus $g$?
\end{question}

\subsection{Obstructions to Explicit Construction}
To resolve Question \ref{MainQue}, it suffices to construct an elliptic surface $X$ over a smooth projective curve of genus $g$ satisfying $\chi(\mathcal O_X) = 1$ with infinite Mordell--Weil group (see also \cite[Proposition 2.5]{Li19}, \cite[Proposition 4.1]{Li26}).

The general strategy for producing infinitely many \((-1)\)-curves of genus $g$ from sections of Jacobian elliptic surfaces is well understood, yet reconciling all geometric constraints simultaneously remains challenging. The primary obstruction arises from the interplay between \(\chi(\mathcal O_X)\), the genus of the base curve, and the existence of a non-torsion section, which guarantees an infinite Mordell--Weil group. 
In general, imposing \(\chi(\mathcal O_X)=1\) restricts the base curve to low genus, while fibrations over higher-genus bases generically admit no non-torsion sections.

We avoid these conflicts by constructing a Jacobian elliptic fibration with compatible Weierstrass model, branching data and base line bundle. 
The point of the construction is to keep $\chi(\mathcal O_X)=1$  while raising the genus of the base, and to produce a section whose non-torsionness can be checked on a single fiber by the Lutz-Nagell theorem \cite{Nagell35, Lutz37}.
\subsection{Main Result}
In this paper, for all $g>1$, we construct the first elliptic surface admitting infinitely many smooth $(-1)$-curves of genus $g$, thereby fully resolving the decade-long open question proposed by Bauer et al.
\begin{theorem}
\label{MainThm}
For every integer $g > 1$, there exists an elliptic surface $X$ carrying infinitely many $(-1)$-curves of genus $g$.
\end{theorem}
\begin{remark}
All prior constructions of infinite families of $(-m)$-curves only achieve either $m>1$ or bounded low base genus; our work provides the first explicit construction valid for arbitrary $g>1$ under the critical constraint $\chi(\mathcal O_X)=1$.
\end{remark}
Beyond settling the open question raised by Bauer et al., our explicit fibration yields a new family of non-isotrivial elliptic surfaces admitting base curves of arbitrarily large genus and infinite Mordell--Weil groups. Note that our examples are Jacobian elliptic surfaces satisfying $\chi(\mathcal O_X)=\kappa(X)=1$. Motivated by these constructions, we conclude by posing the following problem.

\begin{problem}
Classify Jacobian elliptic surfaces $X$ with $\chi(\mathcal O_X)=\kappa(X)=1$.
\end{problem}
\section{Proof of Theorem~\ref{MainThm}}
\subsection*{\bf Notation and Terminology.}
Let $X$ be a smooth projective surface over $\bC$.
\begin{itemize}
\item By a curve on $X$, we mean a reduced and irreducible curve.
\item A negative curve on $X$ is a curve with negative self-intersection.
\item A $(-k)$-curve on $X$ is a curve $C$ with $C^2=-k<0$.
\item $\chi(\mathcal O_X)$ is the Euler characteristic of $X$.
\end{itemize}
\subsection{Basic results}
We collect preliminary definitions, standard facts and auxiliary criteria for elliptic fibrations and Mordell--Weil groups required for our construction. Most material below comes from the foundational monograph \cite{SS19}; we cite separate references for the Nagel--Lutz theorem.
 Throughout the paper, we adopt the following convention to exclude the trivial product case (cf. \cite[Example 5.6]{SS19}) via a standing assumption:
\begin{assumption}
Any elliptic surface $\pi: X\to B$ has a singular fibre.
Thus, we have that $\chi(\mathcal O_X)>0$ by \cite[Corollary 5.50]{SS19}.
\end{assumption}
 \begin{definition}
 Given an algebraic surface fibration $\pi: X\to B$, a \emph{section of $\pi$} is a morphism $\sigma: B \to X$ such that $\pi\circ\sigma$ is the identity map of $B$.
 An elliptic surface $\pi: X\to B$  is \emph{Jacobian} if $\pi$ has a section. 
 Let $\mathrm{MW}(\pi)$ be the Mordell-Weil group of $\pi: X\to B$, i.e., the group of sections, $O$ being the zero section.
 \end{definition}
 \begin{definition}
\cite[Theorem 5.19]{SS19}
Let $E$ be an elliptic curve over the function field $K$ of a smooth projective curve $C$.
Then there exists an elliptic surface $X$ over $C$ whose generic fiber is isomorphic to $E/K$.
 In this case, we say that  $S$ is the Kodaira-N\'eron model of $E/K$.
 We also call $X$ the elliptic surface associated with $E/K$.
 \end{definition}
A fundamental dictionary links global sections of this elliptic surface to $K$-rational points on the elliptic curve over the function field $K$ of $C$.
 \begin{proposition}
\cite[Proposition 5.4]{SS19}
\label{K-rational}
The global sections of $\pi: X\to B$ are in a natural one-to-one correspondence with the $K$-rational points of $E$.
 \end{proposition}
 We next introduce explicit Weierstrass models for elliptic fibrations, which form the core of our construction.
 \begin{definition}[Weierstrass model]
 \label{defn-model}
  Let $L$ be a line bundle on a smooth projective curve $B$. 
Take
\[
a_4\in H^0(B,L^{\otimes 4}),\quad a_6\in H^0(B,L^{\otimes 6}).
\]
We define the Weierstrass equation
\[
y^2=x^3+a_4x+a_6,
\]
whose discriminant
\[
\Delta:=-16\,(4a_4^3+27a_6^2)\in H^0(B,L^{\otimes 12})
\]
is nonzero.
This defines a Weierstrass model of an elliptic fibration $\pi: X\to B$ together with its zero section.
\end{definition}
\begin{remark}
\label{rmk-fibers}
\begin{enumerate}
\item The fiber of $\pi$ over a point $b\in B$ is singular if and only if $\Delta(b)=0$.
 If $\Delta$ vanishes to order $1$ at $b$, then the fiber over $b$ is of Kodaira type $I_1$, and $X$ is smooth along this fiber
(cf. \cite[Section 5.9]{SS19}).
\item  If the Weierstrass data $(L,a_4,a_6)$ is minimal, that is, there exists no point at which $a_4$ vanishes to order at least $4$ and $a_6$ vanishes to order at least $6$, then $\pi:X\to B$ is a relatively minimal Jacobian elliptic fibration with generic fiber $E$, and $L:=(R^1 f_*\mathcal O_X)^\vee$ (cf. \cite[Section 5.10]{SS19}).
\item  Since $f$ has a section, it carries no multiple fibers, and the canonical bundle formula (cf. \cite[Therem 5.44]{SS19})  implies
\[
\omega_X \cong \pi^*(\omega_B \otimes L).
\]
\end{enumerate}      
\end{remark}
We next state two key numerical identities for elliptic surfaces, which are indispensable for our main argument.
 \begin{proposition}
 \cite[Corollary 5.45]{SS19}
 \label{Section-prop}
Let \(\pi : X \to B\) be a Jacobian elliptic surface over a smooth projective curve, and let $P$ be any section of \(\pi\). 
Then \(P^2 = -\chi(\mathcal O_X)\).
 \end{proposition}
 \begin{proposition}
 \cite[Corollary 5.50]{SS19}
 \label{Noether-prop}
 For an elliptic surface $X$, we have 
 $$
 \chi(\mathcal O_X)=\frac{1}{12}e(X)>0.
 $$
 \end{proposition}
 As our central arithmetic technical ingredient, we recall the Nagell--Lutz theorem for bounding torsion section coordinates (cf. \cite{Lutz37, Nagell35}, see also \cite[Corollary VIII.7.2]{Silverman09},  \cite[Theorem 2.5]{ST15}) as follows.
\begin{theorem}
\label{LN-thm}
Let $E/\bQ$ be an elliptic curve with Weierstrass equation
$$
y^2=x^3+Ax+B,\quad A, B\in\bZ.
$$
Suppose that $P\in E(\bQ)$ is a nonzero torsion point.
Then the following statements hold.
\begin{itemize}
\item[(a)] $x(P), y(P)\in \bZ$.
\item[(b)] Either $[2]P=O$ or else $y(P)^2$ divides $4A^3+27B^2$.
\end{itemize}
\end{theorem}
\subsection{Construction}
\label{cons:main}
Fix any integer $g > 1$. 
We first let
\[
D(t):=4(t^2+1)^3+27.
\]

Choose distinct complex numbers $\lambda_1,\dots,\lambda_{2g+1}$, none of which is $0$ or a root of $D$.
Let $B$ be the smooth projective model of the hyperelliptic curve defined by the following equation
\[
w^2=\prod_{i=1}^{2g+1}(t-\lambda_i).
\]
Note that the branch points are $\lambda_1,\cdots, \lambda_{2g+1}$ and $\infty$.
Then $g(B)=g$ by Hurwitz's formula.
Moreover, the hyperelliptic map $\varphi: B\to\bP^1$ is totally ramified at $\infty$, so $t$ has a pole divisor $2\infty$. Let
\[
L:= \mathcal O_B(\infty),\quad \deg L=1 .
\]
As the pole divisor of $t$ is $2\infty$, we have
\[
t\in H^0(B,L^{\otimes 2}),\quad t^2+1\in H^0(B,L^{\otimes 4}),\quad 1\in H^0(B,L^{\otimes 6}).
\]
Therefore, $a_4\coloneqq t^2+1$ and $a_6\coloneqq 1$ are admissible Weierstrass data for $L$ as in Definition \ref{defn-model}.
Let $\pi: X\to B$ be the associated Weierstrass elliptic surface
\[
y^2=x^3+(t^2+1)\,x+1 .
\]
Besides the zero section, $\pi: X\to B$ carries the section
\[
P\coloneqq(x,y)=(0,1),
\]
which is well defined since $1^2=0^3+(t^2+1)\cdot 0+1$.
\subsection{Proof of Theorem \ref{MainThm}}
\begin{lemma}
\label{lem:smooth}
Let $X$ be as in Construction~\ref{cons:main}.
Then there exists a relatively minimal elliptic fibration $\pi:  X\to B$ whose singular fibers are exactly twelve fibers of Kodaira type $I_1$. 
Moreover $\chi(\mathcal O_X)=1$.
\end{lemma}
\begin{proof}
The discriminant of the Weierstrass equation is
\[
\Delta=-16\,(4(t^2+1)^3+27)=-16\,D(t)\in H^0(B,L^{\otimes 12}).
\]
We have $D'(t)=24\,t(t^2+1)^2$, while $D(0)=31$ and $D(\pm i)=27$; hence $D$ and $D'$ have no common root, so the six roots of $D$ are simple.
The branch points of $\varphi: B\to\bP^1$ are $\lambda_1,\dots,\lambda_{2g+1}$ and $\infty$. By the choice of the $\lambda_i$, no root of $D$ is a branch point, so each of the six simple roots of $D$ has two distinct, unramified preimages on $B$. Thus $\Delta$ has twelve simple zeros on the affine part $t\neq\infty$. As $\deg L^{\otimes 12}=12$, these account for all the zeros of $\Delta$; in particular $\Delta(\infty)\neq 0$. Therefore the singular fibers $F_i$ of $f$ are exactly the twelve fibers of type $I_1$ over these twelve points, $i=1,\cdots, 12$.
The data  $(L,a_4,a_6)$ is minimal: $a_6=1$ vanishes only at $\infty$, where $a_4=t^2+1$ does not vanish.
Thus, $a_4$ and $a_6$ never vanish to orders at least $4$ and $6$ respectively.
By Remark \ref{rmk-fibers}, $\pi: X\to B$ is a relatively minimal elliptic fibration.
For each singular fiber $F_i$ ($i=1,\cdots, 12)$, the Euler number $e(F_i)=1$.
Thus, $\chi(\mathcal O_X)=1$ by the Noether formula (cf. Propostion \ref{Noether-prop}).
\end{proof}
\begin{lemma}
\label{lem:nontorsion}
The section $P=(0,1)$ of Construction~\ref{cons:main} is non-torsion.
\end{lemma}
\begin{proof}
Choose $b\in B$ with $t(b)=0$. Since $0$ is neither a branch point nor a root of $D$, the fiber $X_b$ is the smooth elliptic curve
\[
E_0\colon\quad y^2=x^3+x+1
\]
over $\bQ$, and $P$ specializes to the point $P_b=(0,1)\in E_0(\bQ)$. By the duplication formula on $E_0$ (cf. \cite[III.2.3(d)]{Silverman09}), we have
\[
2P_b=\left(\tfrac14,\,-\tfrac98\right).
\]
The model $E_0\colon y^2=x^3+x+1$ has integer coefficients.
So by Theorem~\ref{LN-thm}, every torsion point of $E_0(\bQ)$ has integer coordinates.
Since $2P_b=(\tfrac14,\,-\tfrac98)$ does not have integer coordinates, it is not a torsion point, and hence neither is $P_b$.
As a result, $P$ is a non-torsion section by  Proposition \ref{K-rational}.
\end{proof}

\begin{proof}[Proof of Theorem~\ref{MainThm}]
Let $\pi: X\to B$, $L$, and $P$ be as in Construction~\ref{cons:main}. 
By Lemmas~\ref{lem:smooth} and \ref{lem:nontorsion}, $X$ is an elliptic surface over a curve $B$ of genus $g$ with $\chi(\mathcal O_X)=1$ and a non-torsion section $P$.
Note that $P^2=-1$ by Proposition \ref{Section-prop}.
Then $P$ is a (-1)-curve of genus $g$ since it is a section of $\pi$.
Then $X$ has infinitely many $(-1)$-curve of genus $g$ since $P$ is non-torsion section.
Therefore, we end the proof of Theorem \ref{MainThm}.
\end{proof}
\subsection*{Acknowledgements}
Sichen Li thanks Yi Gu and Sheng-Li Tan for helpful conversations, and Rong Du and Guolei Zhong for constant encouragement.
Jihao Liu was partially supported by the National Key R\&D Program of China \#\allowbreak 2024YFA1014400.
The authors used an AI-assisted tool for heuristic exploration and partial draft sketching during proof development; all mathematical reasoning, verification, and final arguments are entirely the authors' own work.

\end{document}